\documentclass{amsart}
\usepackage{latexsym, amssymb,amsthm,amsmath,verbatim}

\begin{document}

\newtheorem{theorem}{Theorem}[section]
\newtheorem{proposition}[theorem]{Proposition}
\newtheorem{question}[theorem]{Question}
\newtheorem{definition}[theorem]{Definition}
\newtheorem{corollary}[theorem]{Corollary}
\newtheorem{lemma}[theorem]{Lemma}
\newtheorem{claim}[theorem]{Claim}

\renewcommand{\c}{\mathfrak{c}}
\newcommand{\pr}[1]{\left\langle #1 \right\rangle}
\newcommand{\mH}{\mathcal{H}}
\newcommand{\mK}{\mathcal{K}}
\newcommand{\mR}{\mathcal{R}}
\newcommand{\mG}{\mathcal{G}}
\newcommand{\mA}{\mathcal{A}}
\newcommand{\mV}{\mathcal{V}}
\newcommand{\mU}{\mathcal{V}}
\newcommand{\mP}{\mathcal{P}}
\newcommand{\mB}{\mathcal{B}}
\newcommand{\C}{\mathrm{C}}
\newcommand{\mO}{\mathcal{O}}
\newcommand{\mC}{\mathcal{C}}
\newcommand{\D}{\mathrm{D}}
\renewcommand{\O}{\mathrm{O}}
\newcommand{\K}{\mathrm{K}}
\newcommand{\OD}{\mathrm{OD}}
\newcommand{\Do}{\D_\mathrm{o}}
\newcommand{\sone}{\mathsf{S}_1}
\newcommand{\gone}{\mathsf{G}_1}
\newcommand{\sfin}{\mathsf{S}_\mathrm{fin}}
\newcommand{\gfin}{\mathsf{G}_\mathrm{fin}}
\newcommand{\Em}{\longrightarrow}
\newcommand{\menos}{{\setminus}}
\newcommand{\w}{{\omega}}

\title{On a  game theoretic cardinality bound}
\author{Leandro F. Aurichi \and Angelo Bella}
\thanks{This work was done during a visit of the first author to
the University of Catania, sponsored by GNSAGA. The first author
was partially supported by FAPESP 2013/05469-7}
\dedicatory{Dedicated to Ofelia T. Alas on the occasion of her
70th
  birthday}

\maketitle

\begin{abstract} The main  purpose of the paper is the proof of
a cardinal inequality for a space with points  $G_\delta $,
obtained with the help of a long version of the Menger game. This
result improves a similar one  of Scheepers and Tall.
\end{abstract}

\section{Introduction}
Very soon after the publication in 1969 of the celebrated
Arhangel'ski\u\i's cardinal inequality: $|X|\le 2^{\aleph_0}$,
for any first countable Lindel\"of $T_2$ space $X$,  a lot of
attention was  paid to the possibility of extenting this theorem
to the whole class of spaces with points $G_\delta$.  The problem
turned out to be very non-trivial and the first negative
consistent answer was given by Shelah  \cite{shelah}. Later on, a
simpler example of a Lindel\"of $T_3$ space with points
$G_\delta$ whose cardinality is bigger than  the continuum was
constructed by Gorelic \cite{Gorelic2008}. Therefore, it is
interesting to find conditions under which a space with points
$G_\delta$ has cardinality not exceeding $2^{\aleph_0}$. A result
of this kind was obtained by Scheepers and Tall in 2010
\cite{ScheepersTall} with the help of a topological game.    The
main purpose of this note is to strengthen this result.

\section{Main results}
Before giving the announced strengthening of Scheepers-Tall's
inequality, we would like to present a more general consequence
of the hypothesis that player II has a winning strategy in the
long Rothberger game.

A subset $A$ of $X$ is a $G_\kappa $-set if there exists a family
$\mU$ of $\kappa $-many open sets of $X$ such that $A=\bigcap
\mU$.  The
$G_\kappa $-modification $X_\kappa $ of a space $X$ is
obtained by taking as a base the collection of all $G_\kappa
$-sets of $X$.

We use the standard notation for games: we will denote
by $\gone^\kappa(\mA, \mB)$ the game played by player I and
player II such that, for each inning $\xi < \kappa$, player I
chooses $A_\xi \in \mA$. Then player II chooses $a_\xi \in
A_\xi$. Player II wins if $\{a_\xi: \xi < \kappa\} \in \mB$.

We
will denote by $\O$ the family of all open coverings for a given
space. Thus, $\gone^\kappa(\O, \O)$ means that at each inning
player I chooses an open covering and player II chooses one of
its open members. Player II wins if the collection of open sets
chosen forms a covering.

 Thus,  according to this notation, $\gone^\omega(\O,\O) =
\gone(\O, \O)$ is the classic Rothberger game.

In addition,  for a given space $X$, $\D$ will denote the
collection of all families of open sets whose union is dense in
$X$.
Here  no separation axiom is assumed. As usual $\c=2^{\aleph_0}$.

\begin{theorem}
  Let $X$ be a space. If player II has a winning strategy in the
game $\gone^{\w_1}(\O, \O)$, then $L(X_\c) \leq \c$.
\end{theorem}

\begin{proof}
Let $\mG$ be a covering of $X$ by $G_\c$-sets and
for each $G \in \mG$ fix a family $\{U_\beta (G):\beta< \c\}$ of
open subsets of $X$ satisfying
$G=\bigcap\{U_\beta(G):\beta<\c\}$.  Let $F$ be a
winning strategy for player II in $\gone^{\omega_1}(\O,\O)$,
that is a function
$F:\bigcup\{{}^{\alpha +1}\O :\alpha <\omega_1\}\rightarrow
\bigcup \O$, and for any $\phi\in {}^{\alpha +1}\O$ we
have $F(\phi)\in \phi(\alpha )$.

\begin{claim}\label{claim 1}
For any $\alpha <\omega_1$ and any  $\phi\in {}^\alpha \omega$
there exists a point $x_\phi \in X$ such that for each open
neighbourhood $U$ of $x_\phi$ we may find an open covering $\mU$
such that $U=F(\phi\frown \mU)$.
\end{claim}

\begin{proof}
Assume the contrary and for each $x\in
X$ fix an open neighbourhood $U_x$ such that $U_x\ne F(\phi\frown
\mU)$ for every open covering $\mU$. Since the set $\mV=\{U_x
:x\in
X\}$ is an open cover,   we have
$F(\phi\frown \mV)=U_y$ for some $y\in X$. This
contradicts what we are assuming for $y$ and we are done.
\end{proof}

Let us begin by choosing  a point $x_\emptyset $, according to
Claim \ref{claim 1} for $\phi=\emptyset$ and  then choose
$G_\emptyset \in
\mG$ such that $x_\emptyset \in G_\emptyset $. Next,  for each
$\beta<\frak c$ fix an
open covering $\mU_{\{(0,\beta)\}}$ satisfying
$F((0,\mU_{\{(0,\beta)\}}))=U_\beta(G_\emptyset)$. For each
$\beta_0 <\frak c$
choose a point $x_{\{(0,\beta_0)\}}$, according to Claim
\ref{claim 1} for
$\phi=\{(0, \mU_{\{(0,\beta_0)\}})\}$ and choose
$G_{\{(0,\beta_0)\}}\in \mG$ such that $x_{\{(0,\beta_0)\}}\in
G_{\{(0,\beta_0)\}}$. Then,
for each $\beta<\frak c$ fix an open
covering $\mU_{\{(0,\beta_0), (1,\beta)\}}$ satisfying
$F(\{(0,\mU_{\{(0,\beta_0)\}}),(1,\mU_{\{(0,\beta_0),(1,\beta_1)
\}}\})=U_\beta(G_{\{(0,\beta_0)\}})$. At
step $\omega$,
for each $f\in {}^\omega\frak c$ we have already fixed open
covers
$\mU_{f\restriction {n+1}} $, points $x_{f\restriction n}$ and
sets
$G_{f\restriction n}\in \mG$ with $x_{f\restriction n}\in
G_{f\restriction n}$. Then
let $x_f$  be a point as in Claim 1 for $\phi$ defined by
$\phi(n)=\mU_{f\restriction {n+1}}$ and and let $G_f \in \mG$ be
such that $x_f \in G_f$. Then fix open covers
$\mU_{f\frown \beta }$ satisfying $U_\beta(G_f)=F(\phi\frown
\mU_{f\frown
\beta})$.

  By continuing in this manner, for any $f\in {}^\alpha \frak c$
we  choose a point $x_f$, a set $G_f\in \mG$ satisfying
$x_f\in G_f$  and open covers $\mU_{f\frown \beta}$
satisfying $U_\beta(G_f)=F(\phi\frown \mU_{f\frown \beta})$,
where
$\phi(\gamma)=\mU_{f\restriction {\gamma +1}}$ for any
$\gamma <\alpha $. At the end, we have a  collection
$\mH=\{G_f:f\in
\bigcup\{{}^\alpha \frak c :\alpha <\omega_1\}\}$.

\begin{claim}\label{claim 2}
 $\mH$ is a covering of  $X$.
\end{claim}
\begin{proof}
Assume the contrary and  fix a point
$p\in X\setminus \bigcup \mH$.    According to the
hypotheses,

\noindent (*) for each $f\in {}^\alpha \frak c$ we may fix an
ordinal
$\beta_f<\frak c$ in such a way that $p\notin
{U_{\beta_f}(G_f)}$.

By induction, we may define a function $g\in {}^{\omega_1}\frak
c$
such that
$g(0)=\beta_\emptyset$, $g(1)=\beta_{g\restriction 1}$ and in
general $g(\alpha )=\beta _{g\restriction \alpha }$.
 Now, if  player 1 at the $\alpha $-th inning choose
$\mU_{g\restriction {\alpha +1}}$, then because of (*) player II
looses the game. As this is a contradiction, the Claim is proved.
\end{proof}
 Since we obviously have $|\mH|\le \frak c$, the proof of the
Theorem is done.
\end{proof}
 The simpler version of the above  theorem for the classic
Rothberger game provides an alternative proof of  a recent result
already proved  by the first author and Dias.

\begin{corollary} \cite{AurichiDias} Let $X$ be a space. If
player II has a winning strategy in $\gone(\O, \O)$,
then  the $G_\delta$ modification of $X$ is Lindel\"of.
\end{corollary}

Much more relevant for us here is the following:
\begin{corollary}[Scheepers-Tall, \cite{ScheepersTall}]
 If $X$ is a space with points $G_\delta $   and player II has
a winning strategy in the game $\gone^{\omega_1}(\O, \O)$, then
$|X|\le 2^{\aleph_0}$.
\end{corollary}

To appreciate the strength of  the above corollary, notice that
the example of Gorelic \cite{Gorelic2008} provides a space
$X$ with points $G_\delta $ in which player I does not have a
winning strategy in $\gone^{\w_1}(\O, \O)$ and
$|X|>2^{\aleph_0}$ (see \cite{ScheepersTall} for a justification
of this fact).

A very natural question arises on  whether  Scheepers-Tall's
inequality can be improved by replacing $\gone$ with
$\gfin$, \emph{i.e.}, the game where player II chooses finitely
many sets per inning, instead of only one. In other words, we
wonder  whether the long  Menger game can suffice in the above
cardinal inequality.

 We will obtain a positive answer under the continuum
hypothesis CH. To achieve this goal we use  another topological
game, somehow in between $\gone$ and $\gfin$.

\begin{lemma}\label{compact base}
  If $X$ is a space with points $G_\delta $, then for
every compact $K \subset X$ there is a family $\mathcal U$ of
open subsets of $X$ such that $K=\bigcap \mathcal U$ and
$|\mathcal U|\le 2^{\aleph_0}$.
\end{lemma}

\begin{proof}
  First note that each compact $K \subset X$  satisfies $|K|
\leq  2^{\aleph_0}$. This is a consequence of a theorem of
Gryzlov \cite{Gryzlov}.
 For every $x \in K$, let $(V_n^x)_{n \in \w}$ be a family of
open subsets of $X$ satisfying $\bigcap_{n<\w}V_n^x=\{x\}$.

Let $\mB = \{\bigcup_{i = 0}^k
V_{n_i}^{x_i} \supset K: x_0, ..., x_k \in K, n_0, ..., n_k \in
\w\}$. Note that $\bigcap \mB=K$ and $|\mB| \leq  2^{\aleph_0}$.
\end{proof}

\begin{definition}
  We say that an open covering $\mU$ for $X$ is a $\K$-covering
if, for every compact $K \subset X$, there is a $U \in \mU$ such
that $K \subset U$. Let $\K$ be the collection of all
$\K$-coverings.
\end{definition}

\begin{lemma}\label{there is K}
  If $F$ is a winning strategy for player II in the game
$G_1^{\w_1}(\K, \O)$, then for every $(\mU_\alpha)_{\alpha <
\beta}$ sequence of $\K$-coverings for $\beta < \w_1$, there is a
compact $K \subset X$ such that for every open set $U$ such that
$K \subset U$, there is a $\K$-covering $\mU$ such that
$F((\mU_\alpha)_{\alpha < \beta} \smallfrown \mU) = U$.
\end{lemma}

\begin{proof}
  Suppose not. Let $(\mU_\alpha)_{\alpha < \beta}$ such that for
every compact $K \subset X$, there is an open $U_K$ such that $K
\subset U_K$ and for every $\K$-covering $\mU$,
$F((\mU_\alpha)_{\alpha < \beta} \smallfrown \mU) \neq U_K$. Let
$\mU = \{U_K: K \subset X$ is compact$\}$. Note that $\mU$ is a
$\K$-covering. Then there is a compact $K$ such that
$F((\mU_\alpha)_{\alpha < \beta} \smallfrown \mU) = U_K$, which
is a contradiction.
\end{proof}

\begin{theorem}
  Let $X$ be a space with points $G_\delta $. If player II has a
winning
strategy  in the game $G_1^{\w_1}(\K, \O)$ over $X$, then $|X|
\leq  2^{\aleph_0}$.
\end{theorem}

\begin{proof} According to Lemma \ref{compact base},
  for every compact $K \subset X$,  let $(U_\xi^K)_{x < \c}$ be
a  family of open subsets of $X$ such that
$K=\bigcap_{\xi<\c}U_\xi^K$. Let $F$ be a winning strategy for
player II. Let
$K_\emptyset$ be given by Lemma \ref{there is K} such that for
every $\xi < \c$, there is a $\K$-covering $\mU_\xi^\emptyset$
for $X$ such that $F(\mU_\xi^\emptyset) = U_\xi^{K_\emptyset}$.
Let $f: \alpha \Em \w_1$ for some $\alpha < \w_1$. Suppose to
have already
defined $\mU_{f(\beta)}^{f \upharpoonright \beta}$ and $K_{f
\upharpoonright \beta}$ for every $\beta < \alpha$ such that
$F((\mU_{f(\beta)}^{f \upharpoonright \beta})_{\beta < \gamma}) =
U_{f(\gamma)}^{K_{f \upharpoonright \gamma}}$ for every $\gamma <
\alpha$. Let $K_f$ and $(\mU_\xi^f)_{\xi < \c}$ be the open
coverings given by
Lemma \ref{there is K} in such a way that, for every $\xi$,
$F((\mU_{f(\beta)}^{f \upharpoonright \beta})_{\beta < \alpha}
\smallfrown \mU_\xi^f)  = U_{\xi}^{K_f}$.Note that $|\{K_f: f \in
\c^{<w_1}\}| \leq \c$. Therefore, by Gryzlov's Theorem,  $D =
\bigcup_{f \in \c^{<\w_1}}
K_f$  satisfies $|D| \leq \c$.  Thus, to finish the proof  it
is enough to show that $D=X$.

  Suppose not. Then there is a point  $p$ such that $p\notin  D$.
Therefore, there is an $f: \w_1 \Em \c$ such
that $F((\mU_{f(\beta)}^{f \upharpoonright \beta})_{\beta <
\gamma}) = U_{f(\gamma)}^{K_{f \upharpoonright \gamma}}\not\ni p$
for
every $\gamma < \w_1$, since   $p\notin K_{f \upharpoonright
\gamma} $.
But then, playing in this way, player II would loose,
which is a contradiction to the fact that $F$ is a winning
strategy.
\end{proof}
Now, to obtain our main result we need to make use of one more
game.

The compact-open game of length $\kappa $ over a space $X$ is
played as follows:  at the $\alpha $-inning player I chooses a
compact set $K_\alpha $ and player II responds by taking an open
set $U_\alpha \supset K_\alpha$ . The rule of the game is that
player I wins if, and only if, the collection $\{U_\alpha
:\alpha <\kappa \}$ covers $X$.

The following can be obtained by a simple modification of
Galvin's result about the duality of the Rothberger game and the
point-open game (\cite{GalvinGames}):

\begin{lemma} \label{dual}
Let $X$ be a space. Then, for any infinite cardinal
$\kappa $, the games $\gone^\kappa (\K, \O)$ and the
compact-open game of length $\kappa $ are dual. In particular,
player II
has a winning strategy in $\gone^\kappa (\K, \O)$ if and only
if player I has a winning strategy in the compact-open game of
length $\kappa $. \end{lemma}

\begin{theorem}
  Let $X$ be a Tychonoff space. If player II has a winning
strategy
 in the game $\gfin^\kappa(\O, \O)$ for some infinite regular
cardinal
$\kappa$,
then player I has a winning strategy  in the
compact-open game of length $2^{<\kappa }$.
\end{theorem}

\begin{proof}
  Let $\sigma$ be a winning strategy for player II  in
$\gfin^\kappa(\O, \O)$. Let $f:2^{<\kappa }\to {}^{<\kappa}\w$
be a function such that $f(0)=\emptyset $ and  for each $s \in
{}^{<\kappa }\w\setminus \{\emptyset\}$
\begin{equation}\label{eq enumeration}
|f^{-1}(s)|=2^{< \kappa}
\end{equation}

We are going to define a strategy $F$ for player I in the
compact-open game of length
$2^{<\kappa }$ on $X$.
Let $\mC$ be the collection of all open coverings  of $X$.  For
any open subset
$A$ of $X$, fix $A^*$ an open subset of $\beta
X$ such that $A = A^* \cap X$. Define
$$K_0 = \bigcap_{C \in \mC} \overline{\bigcup
\sigma(C)}^{\beta X}.$$
Note that $K_0$   is compact and  $K_0 \subset X$.
We put
$F(0)=K_0 $.  Let
$V_0$ be the answer of player II in  the compact-open
game.  By compactness, there are
$C_0, ..., C_{n_\emptyset }  \in \mC$ such that
$$\bigcap_{i \le  n_\emptyset }  \overline{\bigcup
\sigma(C_i)}^{\beta X} \subset V_0^*.$$
For any $s\in {}^1\w$ let $\alpha _s=\min f^{-1}(s)$ and  put
$C_{f(\alpha _s)}=C_i $ if   $i\le n_\emptyset $ and
$C_{f(\alpha _s)}=\{X\}$ otherwise.

In general,
at the $\beta$ inning of the compact-open game, let $s =
f(\beta)$.

Case 1.  If we have already defined $C_{s\upharpoonright
{\xi+1}}$ for
each $\xi\in dom(s)$ and there are ordinals $\alpha _\xi<\beta$
such that $f(\alpha _\xi)=s\upharpoonright \xi$,  then we put
 $$K_\beta  = \bigcap_{C \in \mC} \overline{\bigcup \sigma((C_{s
\upharpoonright \xi+1})_{\xi \in dom(s)} \smallfrown C)}^{\beta
X}.$$
 Let $V_\beta$ be the answer  of
player II in the compact-open game after player $I$ plays
$F(\beta)=K_\beta$. By
compactness, let $C_{ 0}, ..., C_{ n_s}
\in \mC$ be such that
$$\bigcap_{i\le n_s}   \overline{\bigcup \sigma((C_{s
\upharpoonright \xi+1})_{\xi \in dom(s)} \smallfrown
C_{ i})}^{\beta
X} \subset V_\beta^*.$$

Since at each move we define at most $\omega$ new open coverings,
the set $S$ of all $\alpha <2^{<\kappa }$ for which $C_{f(\alpha
)}$ was already defined has cardinality not exceeding
$|\beta|\omega < 2^{<\kappa }$.

 Therefore, by (\ref{eq enumeration})  for each $i<\omega$ we may
pick $\alpha _i\in  (f^{- 1}(s\smallfrown i)\setminus S)$.   Then
put $C_{f(\alpha_i)}=C_i$ if $i\le n_s$ and $C_{f(\alpha _i)}
=\{X\}$ if $i>n_s$.

  If Case 1  does not take place, then  we simply put
$F(\beta)=K_\beta=\emptyset $ (Case 2).

Let us prove that,  playing according to $F$,  player I always
wins
the
compact-open game. Suppose not and let $x \in X$ be such that $x
\notin \bigcup\{V_\alpha:  \alpha < 2^{< \kappa}\}$,   for a
certain set $\{V_\alpha :\alpha <\kappa \}$ of legitimate
moves of player II. Since
$\bigcap_{i\le n_\emptyset } \overline{\bigcup
\sigma(C_i)}^{\beta X} \subset
V_0^*$, there is an $n_0\le n_\emptyset $ such that $x
\notin \bigcup
\sigma(C_{n_0})$. Then let $C_{\{(0,n_0)\}}=C_{n_0}$.
   Proceeding by induction, assume that for some $\alpha
<\kappa $ we have defined a  function $t\in  {}^\alpha\w$  and
open coverings $C_{t\upharpoonright  {\nu+1}}$, for each $\nu
<\alpha$,  in such a way that $x\notin
\bigcup\sigma((C_{t\upharpoonright \nu+1})_{\nu<\gamma})$ for
each
$\gamma <\alpha $. Moreover, let $\alpha  _\nu<2^{<\kappa }$ be
such that $f(\alpha _\nu)=t\upharpoonright  \nu$ for each
$\nu<\alpha $.
 Since $cf(2^{<\kappa })\ge cf (\kappa )=\kappa
$ and \ref{eq enumeration} holds, we may pick  $\beta\in  f^{-
1}(t)$ such that $\alpha _\nu<\beta$ for each $\nu<\alpha $.
 According to our
construction, Case 1   holds and so there is an integer $j\le
n_t$ such that $x\notin
\bigcup\sigma((C_{t\upharpoonright \nu+1})_{\nu<\alpha
}\smallfrown C_{t\smallfrown j})$. This extents $t$ to a function
with domain $\alpha +1$ and  the induction is complete.
 At the end, we   obtain a function $t\in {}^\kappa \w$ and open
coverings $c_{t\upharpoonright \nu+1}$ for each $\nu<\kappa $,
in such a way that the play
$$C_{\{(0,n_0)\}}, \sigma(C_{\{(0,n_0)\}}), \ldots,
 C_{t\upharpoonright \nu+1},
\sigma((C_{t\upharpoonright \xi+1})_{\xi\le \nu}), \ldots  $$
is lost by player II, in evident contradiction   with the
fact that $\sigma$
is a winning strategy.
\end{proof}

We wish to thank R. Dias and the careful referee for the great
help in the previous proof.

Now, by the above theorem and Lemma \ref{dual},  we easily get
the result mentioned in the abstract.

\begin{corollary} [CH] \label{ch}
Let $X$ be a Tychonoff space with points $G_\delta$.
If player II has a winning strategy in the game
$\gfin^{\w_1}(\O,\O)$, then $|X|\le  2^{\aleph_0}$.
\end{corollary}
As a further corollary, we get a more direct proof of the
following
result.

\begin{corollary}[Telgarsky, \cite{Telgarsky1984}]
  Let $X$ be a  Tychonoff space. Then player II has a winning
strategy  in the game $\gfin(\O, \O)$ if, and only if, player
II has a winning strategy  in the game $\gone(\K, \O)$.
\end{corollary}

Also, if we assume the continuum hypothesis, then we can go up to
$\w_1$:

\begin{corollary}[CH] Let $X$ be a Tychonoff space. Then player
II has a winning
strategy  in the game $\gfin^{\w_1}(\O, \O)$ if, and only if,
player II has a winning strategy  in the game $\gone^{\w_1}(\K,
\O)$.
\end{corollary}

Further  game theoretic cardinality bounds can be found in
\cite{bellaspadaro}. In  particular, Theorem 2.2 of
\cite{bellaspadaro} provides a version of Scheepers-Tall's
inequality for the game $\gone^{\omega_1}(\O,\D)$  in the class
of first countable regular spaces.  Although not all  proofs of
the  results presented here before Corollary \ref {ch} have a
direct analogous by passing from ``$(\O, \O)$'' to ``$(\O,
\D)$'', we
believe the following question  could have a positive answer:
\begin{question} Let $X$ be a first countable regular space and
assume that player II has a winning strategy in the game
$\gfin^{\omega_1}(\O, \D)$. Is it true that $|X|\le
2^{\aleph_0}$? \end{question}

\section{Games and open neighborhood assignments}

We end this paper showing some results that split the local parts
from the global parts in some variations of the games presented
above. For the global parts we use the concept of open
neighborhoods assignments:

\begin{definition}
  Let $X$ be a topological space. We say that a family $(V_x)_{x
\in X}$ is an {\bf open neighborhood assignment} for $X$ if each
$V_x$ is an open set such that $x \in V_x$.
\end{definition}

The key idea for the next game is that we will not ask for a
dense set at the end, but for something that looks like a dense,
from the point of view of a given open neighborhood assignment:

\begin{definition}
  Let $X$ be  a space and let $(V_x)_{x \in X}$ be an
open neighborhood assignment. Define the game $G((V_x)_{x \in
X})$ as follows. For every inning $\xi < \w_1$, player I
chooses an open covering $\mC_\xi$ for $X$. Then, player II
chooses $C_\xi \in \mC_\xi$. We say that player II wins the game
if for every $x \in X$ there is a $\xi < \w_1$ such that $V_x
\cap C_\xi \neq \emptyset$.
\end{definition}

\begin{proposition}
If $X$ is a first countable space such that player II has a
winning strategy  in the game $G((V_x)_{x \in X})$ for every open
neighborhood assignment $(V_x)_{x \in X}$, then player II has a
winning strategy  in the game $\gone^{\w_1}(\O, \D)$.
\end{proposition}

\begin{proof}
  For every $x \in X$, let $(V_n^x)_{n \in \w}$ be a local base
 at $x$. For each $n \in \w$, let $\sigma_n$ be a winning
strategy  in the game $G((V_n^x)_{x \in X})$. Let us define a
strategy   for player II in the $\gone^{\w_1}(\O, \D)$. In the
first
inning, player II plays following $\sigma_0$. Then, at inning $n
\in \w$, player II plays following $\sigma_n$, pretending that
this is the first inning. For each limit ordinal $\xi < \w_1$,
player II plays following $\sigma_0$, considering only the
previous  moves where $\sigma_0$ was used. Then, for $\xi + n$,
player II plays following $\sigma_n$, considering only the
previous  moves where $\sigma_n$ was used.

  Let us show that this is a winning strategy. Suppose not. Then
there is an $x \in X$ such that $x \notin \overline{\bigcup_{\xi
< \w_1} C_\xi}$, where $C_\xi$ is the open set choose by II in
the $\xi$-th inning. Then, there is an $n \in \w$ such that
$V_n^x \cap \bigcup_{\xi < \w_1} C_\xi = \emptyset$. This is a
contradiction, since there is a limit ordinal $\xi < \w_1$ such
that $C_{\xi + n} \cap V_n^x \neq \emptyset$ because $\sigma_n$
is a winning strategy.
\end{proof}

It may look that finding a winning strategy for player II in the
$G((V_x)_{x \in X})$ is much easier then finding a winning
strategy for player II in $\gone^{\w_1}(\O, \D)$. We will show
that
in two of the most simple cases, it just does not make any
difference.

\begin{definition}
  Let $X$ be a topological space. We call the {\bf (open
neighborhood assignment)-weight of $X$} (ona-$w(X)$) the least
cardinal $\kappa$ such that for every open neighborhood
assignment $(V_x)_{x \in X}$, there is an open neighborhood
assignment refinement $(W_x)_{x \in X}$ (\emph{i.e.}, for every
$x$, $x \in W_x \subset V_x$) such that $|\{W_x: x \in X\}| \leq
\kappa$.
\end{definition}

\begin{proposition}
  Let $X$ be a topological space. Then $w(X) =$ ona-$w(X)
\chi(X)$.
\end{proposition}

\begin{proof}
   Trivially, ona-$w(X) \chi(X) \leq w(X)$. For every $x \in X$,
let $(\mV_\xi^x)_{\xi < \chi(X)}$ be a local base for $x$. Then,
for every $\xi < \chi(X)$, let $(W_\xi^x)_{x \in X}$ be an open
neighborhood assignment refinement of $(V_\xi^x)_{x \in X}$ such
that $|\{W_\xi^x: x \in X\}| \leq$ ona-$w(X)$. Note that $B =
\bigcup_{\xi < \chi(X)} \{W_\xi^x: x \in X\}$ is such that $|B|
\leq$ ona-$w(X)\chi(X)$. We will show that $B$ is a base for $X$.
Let $V$ be an non-empty set. Let $x \in V$. Then there is an
$V_\xi^x \subset V$. Thus, $x \in W_\xi^x \subset V$.
\end{proof}

\begin{corollary}
  If $X$ is a first countable space, $w(X) =$ ona-$w(X)$.
\end{corollary}

\begin{definition}
  Let $X$ be a topological space. We call the {\bf (open
neighborhood assignment)-density of $X$} (ona-$d(X)$), the least
cardinal $\kappa$ such that for every $(V_x)_{x \in X}$ open
neighborhood assignment, there is a subset $D \subset X$ such
that $|D| \leq \kappa$ and $D \cap V_x \neq \emptyset$ for every
$x \in X$.
\end{definition}

\begin{proposition}
  Let $X$ be a topological space. Then $d(X) \leq$
ona-$d(X)\chi(X)$.
\end{proposition}

\begin{proof}
  For each $x \in X$, let $(V_\xi^x)_{\xi < \chi(X)}$ be a local
base for $x$. For every $\xi < \chi(X)$, let $D_\xi \subset X$ be
such that $|D_\xi| \leq$ ona-$d(X)$ and $D_\xi \cap V_\xi^x \neq
\emptyset$ for every $x \in X$. Note that $D = \bigcup_{\xi <
\chi(X)} D_\xi$ is such that $|D| \leq$ ona-$d(X)\chi(X)$. We
will show that $D$ is dense. Let $V$ be a non-empty open set. Let
$x \in V$. Let $\xi < \chi(X)$ such that $V_\xi^x \subset V$.
Note that $D_\xi \cap V_\xi^x \neq \emptyset$.
\end{proof}

\begin{corollary}
  If $X$ is a first countable space, then $d(X) =$ ona-$d(X)$.
\end{corollary}

\bibliographystyle{abbrv}

\def\cprime{$'$}

\end{document}